\documentclass{amsart}

\usepackage{a4wide,amsmath,amssymb}
\usepackage[toc,page]{appendix}

\numberwithin{equation}{section}

\newtheorem{thm}{Theorem}[section]

\newtheorem{cor}[thm]{Corollary}
\newtheorem{rem}{Remark}[section]
\newtheorem{de}{Definition}[section]
\newtheorem{ex}{Example}[section]

\def\eps{\varepsilon}
\def\a{\alpha}
\def\de{\partial}
\def\d{\delta}
\newcommand{\ol}{\overline}
\newcommand{\ul}{\underline}
\newcommand{\Z}{{\mathbb Z}}

\newcommand{\R}{{\mathbb R}}

\begin{document}

\title[Liouville properties and critical value of fully nonlinear elliptic operators]{
Liouville properties and critical value\\
of fully nonlinear elliptic operators}
\author[M. Bardi, A. Cesaroni]{Martino Bardi 
 {and}  Annalisa Cesaroni 
}
 \thanks{
 The authors  are members of the Gruppo Nazionale per l'Analisi Matematica, la Probabilit\`a e le loro Applicazioni (GNAMPA) of the Istituto Nazionale di Alta Matematica (INdAM), and are partially supported by the research project of the University of Padova "Mean-Field Games and Nonlinear PDEs".
}
\address{Department of Mathematics, University of Padova, Via Trieste 63, 35121 Padova, Italy} \email{bardi@math.unipd.it}
\address{Department of Statistical Sciences,  University of Padova, via Cesare Battisti  241/243,  35121 Padova, Italy} \email{annalisa.cesaroni@unipd.it}
\begin{abstract} 
We prove some Liouville properties for sub- and supersolutions of fully nonlinear degenerate elliptic equations in the whole space. Our assumptions allow the coefficients of the first order terms to be large at infinity, provided they have an appropriate  sign, as in Ornstein-Uhlenbeck operators. We give two applications. The first is a stabilization property for large times of solutions to fully nonlinear parabolic equations. The second is the solvability of an ergodic Hamilton-Jacobi-Bellman equation that identifies a unique critical value of the operator.
\end{abstract}
\subjclass{35B53, 
35B40  
35J70   	
35J60   
49L25   	
}
\keywords{Liouville property, fully nonlinear PDEs, degenerate elliptic PDEs, stabilization in parabolic equations, ergodic Hamilton-Jacobi-Bellman equations, viscosity solutions.}
\maketitle

\section{Introduction}
We consider fully nonlinear degenerate elliptic partial differential equations
\begin{equation}
F(x,u, Du, D^2u) =0 , \qquad\text{in } \R^N ,
\label{F0}
\end{equation}
within the theory of viscosity solutions, and look for sufficient conditions for the validity of Liouville type results such as
\begin{multline} 
\text{any subsolution (respectively, supersolution) of \eqref{F0}
}\\\text{ bounded from above (respectively, from below) is a constant.}
\label{L}
\end{multline} 
For solutions of the equation $F(D^2u) =0$ with $F$ uniformly elliptic, the result follows from the Harnack-type inequalities 
 for such PDEs \cite{CC}. For subsolutions, however, such inequalities do not hold and different tools must be used.
 Cutr\`i and Leoni \cite{CL} proved \eqref{L} for subsolutions of 
 equations of the form $F(x, D^2u) + h(x)u^p =0$ by a nonlinear extension of the Hadamard three spheres theorem. Capuzzo Dolcetta and Cutr\`i \cite{CDC} and Chen and Felmer \cite{CF} studied inequalities 
 with $F$ of the general form \eqref{F0}, where the dependence on the first order derivatives 
  is a nontrivial difficulty that is overcome if  the coefficients multiplying $Du$ decay at infinity in a suitable way. 
  All these papers assume the uniform ellipticity of $F$. 
 
 Our approach is different and requires different assumptions. It is inspired by our paper \cite{BCM} on a linear equation of Ornstein-Uhlenbeck type modelling stochastic volatility in finance. We suppose the existence of a sort of Lyapunov function $w$ for the operator $F$, namely, a supersolution of \eqref{F0} for $|x|>R_o$, for some $R_o$, and such that $w\to +\infty$ as $|x|\to +\infty$. We search examples of such functions among radial ones. For instance, in the case of a Hamilton-Jacobi-Bellman operator
 \[
 \inf_{\a\in A} \{-tr (a(x,\a) D^2u) - b(x,\a)\cdot Du + c(x,\a) u\}
  \]
 $w(x)=|x|^2$ is a  Lyapunov function if
\[
\sup_{\a\in A} (tr\, a(x,\a) + b(x,\a)\cdot x - c(x,\a)|x|^2/2) \leq 0 \quad\text{ for } |x|\geq R_o .
\]
If $\inf c >0$ this allows for quite general coefficients $a, b$, whereas for $c\equiv 0$ it is satisfied by a drift $b$ of the kind appearing in Ornstein-Uhlenbeck operators plus a possible perturbation of lower order, such as
\[
b(x,\a)=\gamma(m-x) +\tilde b(x,\a) , \qquad \gamma >0, \qquad \lim_{x\to\infty} \sup_{\a\in A}\frac{\tilde b(x,\a)\cdot x}{|x|^2}=0 .
\]
Note that here it is helpful that $b$ is large for $|x|$ large, provided 
it has the appropriate sign, that is, it
points towards the origin, whereas  in \cite{CDC, CF} $b$ must be small at infinity.

The other main 
ingredient of our method is a strong maximum principle for the equation  \eqref{F0}. This is true in the uniformly elliptic case,  but also for several degenerate operators, see \cite{BDL1, BDL}. Our main example is a quasilinear operator whose principal part is hypoelliptic in H\"ormander's sense. This seems to be the first Liouville-type result for subelliptic inequalities with nonlinearities involving $Du$. We refer to Chapter 5.8 of the monograph \cite{BLU} and the references therein for a survey about Liouville properties for sublaplacians, mostly obtained by Harnack-type inequalities for solutions. We refer also to \cite{CDC97} for results on inequalities of the form $Lu + h(x)u^p\leq 0$ with $L$ linear degenerate elliptic, and to \cite{LK09, LK15, MMT} for more recent results on linear subelliptic equations. 

For uniformly elliptic $F$ with constants $0<\lambda\leq\Lambda$ we exploit the comparison with Pucci maximal and minimal operators $\mathcal M^+, \mathcal M^-$ associated to $\lambda, \Lambda$ and the Lyapunov function $w(x)=\log|x|$. If 
\[
 F(x,t, p, X) \geq \mathcal M^-(X) +  \inf_{\a\in A} \{c(x,\a)t - b(x,\a)\cdot p\} ,
 \]
 we prove the Liouville property for subsolutions under the assumption
\[
\sup_{\a\in A} ( b(x,\a)\cdot x - c(x,\a)|x|^2\log(|x|)) \leq  \lambda - (N-1)\Lambda  \quad\text{ for } |x|\geq R_o.
\]
This e some results in  \cite{CDC}. Let us mention that other results on 
Liouville type properties for fully nonlinear equations are in the papers \cite{BD, ARV, R, AS, PP}.

The second part of the paper is devoted to two applications of the Liouville properties, both for uniformly elliptic $F$. The first is the stabilization in the space variables for large times of solutions to the parabolic equation
\[
u_t + F(x,
 Du, D^2u)=0 \quad \text{ in  }[0, +\infty)\times \R^N, \qquad
u(0,x) = h(x)  \quad   \text{ in  }\R^n,
\]
with $F$ positively 1-homogeneous in $(p,X)$ and $h \in BUC(\R^N)$. We prove that
$$
\limsup_{t\to +\infty} u(t,x)=\ol u \qquad \text{and} \qquad 
\liminf_{t\to +\infty} u(t,x)=\ul u
$$
 are constant, a result previously known for $F$ and $h$ periodic in $x$ (and in such a case $\ol u=\ul u$ and the convergence is uniform). 
 The stabilization to a constant $\ol u=\ul u$ has been studied by several authors for linear equations under additional conditions on $h$ (see \cite{EKT} and the references therein),  and it is known  that even for the heat equation it can be $\ol u>\ul u$ for some bounded and smooth $h$ \cite{CE}.

The second application concerns the so-called ergodic HJB equation
 \begin{equation*}
 \inf_{\alpha\in A} \{-tr\, a(x,\a)D^2\chi(x) + b(x,\a)\cdot D\chi(x) -l(x,\a)\}= c,\qquad   x\in\R^N,
\end{equation*}
where the unknowns are the critical value $c\in \R$ and $\chi\in C(\R^N)$ that must also satisfy a growth condition as $|x|\to\infty$. This problem arises in ergodic stochastic control (see, e.g., \cite{AL, ABG, ks, Ichi, Cir} and the references therein), weak KAM theory in the 1st order case $a\equiv 0$ \cite{LPV, Fat}, periodic homogenization \cite{LPV, Ev92, ABM}, singular perturbations \cite{BA2, BCM, bc}, and long-time behavior of solutions to non-homogeneous parabolic equations (see, e.g., \cite{bs, fil, IS, CFP} and the references therein). The Liouville property plays a crucial role in the proof of the uniqueness of $\chi$, up to additive constants, and of $c$. The existence of the solution pair is proved by the vanishing discount approximation and using the Krylov-Safonov H\"older estimates, 
as in \cite{Ev92, AL, BA2} for the periodic case and 
 in \cite{BCR, CCR} for HJB equations degenerating at the boundary of a bounded open set.  
The case of a semilinear uniformly elliptic equation in the whole space, under some dissipativity condition, has been considered in \cite{ks, Ichi}, see also references therein. To our knowledge our result is the first for fully nonlinear elliptic equations in the whole $\R^N$ without any periodicity assumption obtained by PDE methods. See  \cite{ABG, CFP} for probabilistic results under different conditions.

The paper is organized as follows. In Section \ref{sect:HJB}  we prove Liouville properties a bit more general than \eqref{L} (the sub- and supersolution can be unbounded, provided they are controlled at infinity by the Lyapunov function) for possibly degenerate HJB operators, and then refine the results for Pucci extremal operators plus lower order terms. Section \ref{sect:ue} is devoted to general uniformly elliptic operators and Section \ref{sect:hyp} to quasilinear HJB inequalities with hypoelliptic principal part. In Section \ref{sect:lts} we study the stabilization in space for large times of solutions to fully nonlinear parabolic equations. Finally, Section \ref{sect:erg} deals with the unique solvability of the ergodic HJB equation.

\section{Hamilton-Jacobi-Bellman operators}\label{sect:HJB}
\subsection{General HJB operators}
We begin with the concave operators
\begin{equation}\label{HJB}
L^\a u:=tr (a(x,\a) D^2u) + b(x,\a)\cdot Du , \qquad G[u]:=\inf_{\a\in A} \{-L^\a u + c(x,\a) u\}, 
\end{equation}
where the coefficients $a, b, c$ are defined in $\R^N\times A$ and are at least continuous,  
 $A$ is a metric space, and $tr$ denotes the trace. Throughout the paper sub- and supersolutions are meant in the viscosity sense.
We assume the following conditions: \begin{enumerate}
\item \label{CP}
$F(x,t,p,X)=\inf_{\a\in A} \{-tr (a(x,\a)X) - b(x,\a)\cdot p + c(x,\a) t\}$ is continuous in $\R^N\times\R\times\R^N\times S^N$, $c\geq 0$, and $G$  satisfies the Comparison Principle in any bounded open set $\Omega$, 
i.e., if $u, v$ are, respectively, a sub- and a  supersolutions of $G[u]=0$ in 
$\Omega$ and $u\leq v$ on $\de\Omega$, then $u\leq v$ in 
 $\Omega$;
\item \label{SMP} $G$ satisfies the Strong Maximum Principle, i.e., any viscosity subsolution in $\R^N$ that attains an interior nonnegative maximum must be constant;
\item \label{w} there exist $R_o\geq 0$ and $w\in LSC(\R^N)$ such that $G[w]\geq 0$  for $|x|>R_o$ and $\lim_{|x|\to\infty} w(x)=+\infty$.
\end {enumerate}

Sufficient conditions for \eqref {CP} are well-known and will be recalled 
later in this section (a general reference is \cite{CIL}).
Sufficient conditions for \eqref{SMP}  can be found in \cite{BDL1, BDL},  we will use the strict ellipticity 
 \eqref{C3}  in this section 
 and some form of hypoellipticity in Section \ref{sect:hyp}.

\begin {thm}\label{teoHJB}Assume \eqref{CP}, \eqref{SMP}  and \eqref{w}. Let  $u\in USC(\R^N)$ satisfy  $G[u]\leq 0$ in $\R^N$ 
and \begin{equation}
\label{upperbound}
\limsup_{|x|\to\infty} \frac{u(x)}{w(x)}\leq 0.
\end{equation}
If either $u\geq 0$ or $c(x,\a)\equiv 0$ for every $x,\a$, then $u$ is constant. 
\end{thm}
\begin{proof}
We divide the proof in various steps. 

{\bf Step 1} . Define $u_\eta(x):=u(x)-\eta w(x)$, for $\eta>0$.  
Possibly increasing $R_o$ we can assume that  $u$ is not constant in the  ball  $\{x\ |\ |x|\leq R_o\}$, otherwise we are done. 
Set
\[
C_\eta:=\max_{|x|\leq R_o} u_\eta(x) .
\]

First of all we show that under our assumptions, $G[C_\eta]\geq 0$ for every $\eta$ sufficiently small. 

Indeed, if $c(x,\alpha)\equiv 0$ then necessarily $G[C_\eta]=0$. 
On the other hand, if $c\not\equiv 0$ then $u\geq 0$ and in this case   we can assume that for $\eta$ sufficiently small $C_\eta>0$. 
In fact, if this were not the case, we could conclude letting $\eta\to 0$ that $u(x)=0$ for every $|x|\leq R_o$, in contradiction with the fact that we assumed that $u$ is not constant 
in the  ball  $\{x\ |\ |x|\leq R_o\}$.

{\bf Step 2} 
The growth condition  \eqref{upperbound} implies that 
$\limsup_{|x|\to\infty} \frac{u_\eta(x)}{w(x)}\leq -\eta<0$ for all $\eta>0$, so $\lim_{|x|\to\infty} u_\eta(x)=-\infty$.
Then for all $\eta>0$ there exists $M_\eta>R_o$ such that
\[
u_\eta(x) \leq C_\eta \quad \text{ for all } |x|\geq M_\eta.
\]

{\bf Step 3} . We prove that for all $\eta>0$, $u_\eta$ satisfies $G[u_\eta]\leq 0$   in $\{x\ | |x|>R_o\}$. 

Fix  $\overline{x}$, $|\overline{x}|>R_o$, 
and a smooth function $\phi$ such that $u_\eta(\ol{x})-\phi(\ol{x})=0$ and $u_\eta-\phi$ has a strict maximum at $\overline{x}$. 

Assume by contradiction that $G[\phi(\ol{x})]>0$.  Let $\delta>0$ sufficiently small  such that $G[\phi(\ol{x})-\delta]>0$.  Indeed since  $G[\phi(\ol{x})]=F(\ol{x}, \phi(\ol{x}), D\phi(\ol{x}), D^2\phi(\ol{x}))>0$,  by continuity of $F$, stated in assumption   \eqref{CP}, there exists $\delta>0$ such that 
 $F(\ol{x}, \phi(\ol{x})-\delta, D\phi(\ol{x}), D^2\phi(\ol{x}))>0$. So, using again continuity of $F$ and regularity of $\phi$  we get that there exists $0<r<|\ol{x}|-R_o$ such that  $G[\phi-\delta]>0$ in $B(\ol{x}, r)$.   

Since $u_\eta-\phi$ has a strict maximum at $\ol{x}$, there exists $0<k<\delta$ such that $u_\eta-\phi\leq-k<0$ on $\partial B(\ol{x},r)$.

Moreover, we claim that  $\eta w+\phi-k$ satisfies 
  $G[\eta w+\phi-k]\geq 0$  in $B(\ol{x},r)$. Indeed take $\tilde{x}\in B(\ol{x},r)$ and $\psi$ smooth such that $\eta w+\phi-k-\psi$ has a minimum at $\tilde{x}$.  Using the fact that $w$ is a viscosity supersolution, \eqref{w},   we get
   \begin{eqnarray*} 
0&\leq&  G[\psi(\tilde{x})-\phi(\tilde{x})+k]=\inf_{\a\in A} \{-L^\a \psi(\tilde{x}) + L^\a \phi(\tilde{x})+c(\tilde{x},\alpha) (\psi(\tilde x)-\phi(\tilde x)+k)\} \\ & \leq&  \inf_{\a\in A} \{-L^\a \psi(\tilde{x})+c(\tilde{x},\alpha) \psi(\tilde x)\}-\inf_{\a\in A} \{-L^\a \phi(\tilde{x})+c(\tilde{x},\alpha) (\phi(\tilde x)-\delta)\}\\&=&G[\psi(\tilde{x})]-G[\phi(\tilde{x})-\delta]<G[\psi(\tilde{x})] .
\end{eqnarray*} 
Therefore $G[\psi(\tilde{x})]\geq 0$, which implies that $G[\eta w+\phi - 
k]\geq 0$  in $B(\ol{x},r)$.

Since $u\leq \eta w+\phi -k$ on $\partial B(\ol{x},r)$, we can now apply the Comparison Principle and get  $u\leq \eta w+\phi -k$ in $B(\ol{x}, r)$, in contradiction with the fact that $u(\ol{x})=\eta w(\ol{x})+\phi(\ol{x})$.

{\bf Step 4} .  Now we use the Comparison Principle in $\Omega=\{x : R_o<|x|<M_\eta\}$. 
Since $G[C_\eta]\geq 0$, by Step 1,  we get $u_\eta(x) \leq C_\eta$ in $\Omega$, using Step 2 and 3. Therefore 
\[
u_\eta(x) \leq C_\eta \quad \text{ for all } |x|\geq R_o .
\]
By letting $\eta\to 0^+$ we obtain 
that 
\[
u(x)\leq \max_{|y|\leq R_o} u(y) 
 \]
and then $u$ attains its maximum $\ol{x}$ over $\R^N$. 
Now if $u\geq 0$ the Strong Maximum Principle gives the desired conclusion. If, on the other hand $c(x,\a)\equiv 0$, then we substitute $u$ with $u+|u(\ol{x})|$ and we conclude. 
\end{proof}

Now we turn to the study of Liouville properties for  supersolutions and 
 we substitute assumption \eqref{SMP} with the following 
 \begin{enumerate}\item[$(2')$]  $G$ satisfies the Strong Minimum Principle, i.e., any viscosity supersolution in $\R^N$ that attains an interior nonpositive minimum must be constant.
\end{enumerate} 
\begin{rem}
\upshape
  Let $v\in LSC(\R^N)$ satisfy $G[v]\geq 0$ 
and 
$\liminf_{|x|\to\infty} \frac{v(x)}{w(x)}\geq 0.$
Assume \eqref{CP}, $ (2')$ and \eqref{w} where the condition $G[w]\geq 0$ can be replaced by the much weaker requirement that  $-L^\a w + c(x,\a) w\geq 0$ for some $\a$. Then an argument similar to the proof of Theorem \ref{teoHJB} gives that,
if either $v\leq 0$  or $c(x,\a)\equiv 0$ for every $x,\a$,  then $v$ is a constant. 
 \end{rem}
 %
Consider convex operators of the form 
\begin{equation}
\label{HJB2}
 \tilde G[u]:=\sup_{\a\in A} \{-L^\a u + c(x,\a) u\},
 \end{equation} where $L^\a$ is as in \eqref{HJB}. 
The assumption \eqref{w} will be replaced with the following 
 \begin{equation}
 \label{wconv}
 \text{ 
$\exists R_o\geq 0$ and $W\in USC(\R^N)$ such that $\tilde G[W]\leq 0$  for $|x|>R_o$  and $\lim_{|x|\to\infty} W(x)=-\infty$.}
\end{equation}
\begin {thm}
\label{convex} 
Assume \eqref{CP}, $(2')$, \eqref{wconv},  and let $v\in LSC(\R^N)$ be a supersolution to $\tilde G[v]\geq 0$ in $\R^N$, such that 
 \begin{equation}
 \label{nuova}
\liminf_{|x|\to +\infty} \frac{v(x)}{W(x)}\leq 0.
\end{equation}
 If 
either $v\leq 0$ or $ c(x,\a)\equiv 0$, then  $v$ is constant.
\end{thm} 
 \begin{proof}  We consider the function $v_\eta(x)=v(x)-\eta W(x)$. 
 As in Step 1 of the proof of Theorem \ref{teoHJB}, we get that $\tilde G[c_\eta]\leq 0$ for $\eta$ sufficiently small, where 
 $$c_\eta:=\min_{|x|\leq R_o} v_\eta(x) .
 $$
  Following Step 2, by \eqref{nuova} and $W(x)<0$ for $|x|$ large, we get 
 $\lim_{|x|\to\infty} v_\eta(x)=+\infty$  for all $\eta>0$.
Then for all $\eta>0$ there exists $M_\eta>R_o$ such that
\[
v_\eta(x) \geq c_\eta \quad \text{ for all } |x|\geq M_\eta.
\]
Moreover, arguing as in Step 3 of the same proof, it is possible to 
show that  $\tilde G[v_\eta]\geq 0$ for $|x|>R_o$. So, as in Step 4, we   use the Comparison Principle 
to conclude
that $v_\eta(y)\geq c_\eta
 $ for $|y|\geq R_o$. 
Finally, we let $\eta\to 0$ and get $v(y)\geq \min_{|x|\leq R_o} v(x) $ for $|y|\geq R_o$. From this  we deduce, using  the Strong Minimum Principle, that $v$ is constant. 
 \end{proof}

Let us recall some standard 
 conditions on the coefficients of $L^\a$ that imply \eqref {CP}, \eqref{SMP} and $(2')$: $a(x,\a)=\sigma(x,\a)\sigma(x,\a)^T$ for some $N\times m$ matrix-valued function $\sigma$ and 
\begin{multline}
\label{C1}
\forall R>0 \, \exists K_R \text{ such that } \sup_{|x|\leq R}(|\sigma|+|b| +|c|)\leq K_R ,\\ 
\sup_{|x|,|y|\leq R, \a\in A}(|\sigma(x,\a) - \sigma(y,\a)| + |b(x,\a) - b(y,\a)|
)
\leq K_R |x-y| ;
\end{multline}
\begin{equation}
\label{C2}
c(x,\a)\geq 0 \text{ and } c \text{ is continuous in } x \text{ uniformly in } |x|\leq R , \a\in A ; 
\end{equation}
\begin{equation}
\label{C3}
\xi^T a(x,\a) \xi\geq |\xi|^2/K_R \quad \forall \xi\in\R^N , \, |x|\leq R .
\end{equation}

\begin{cor}
\label{cor1}
Assume the operators $L^\a$ satisfy \eqref{C1}, \eqref{C2}, \eqref{C3}, and
\begin{equation}
\label{C4}
\sup_{\a\in A} (tr\, a(x,\a) + b(x,\a)\cdot x - c(x,\a)|x|^2/2) \leq 0 \quad\text{ for } |x|\geq R_o .
\end{equation}
Let $u\in USC(\R^N)$ be a viscosity subsolution to   $G[u]\leq 0$ such that $\limsup_{|x|\to +\infty} \frac{u(x)}{|x|^2}\leq 0$. Assume either that $u\geq 0$ or $c(x,\a)\equiv 0$, then $u$ is a constant.\\
 Let $v\in LSC(\R^N)$be a viscosity supersolution to   $\tilde G[v]\geq 0$ such that $\liminf_{|x|\to +\infty} \frac{v(x)}{|x|^2}\geq 0$.  Assume either that $v\leq 0$ or $c(x,\a)\equiv 0$, then $v$ 
 is a constant.
\end{cor}
\begin{proof} Note that $G$, $\tilde G$ are uniformly elliptic in any bounded set by \eqref{C3}. Then \eqref{C1}, \eqref{C2}, and \eqref{C3} imply the Comparison Principle on bounded sets \eqref{CP}, see \cite{J} or \cite{BM}. Moreover the Strong Maximum Principle \eqref{SMP} for $G$ and the    the Strong Minimum Principle $(2')$  for $\tilde G$ hold by Corollary 2.7 of \cite{BDL}. 

Next we check the properties \eqref{w} by choosing   $w(x) =|x|^2/2$. 
Since
\[
L^\a w = tr\, a(x,\a) + b(x,\a)\cdot x  
\]
 \eqref{C4} implies 
 $\inf_{\a\in A}\{-L^\a w+c(x,\a)|x|^2/2\} \geq 0$  for $ |x|\geq R_o$. Note that choosing $W(x)=-|x|^2 /2$, the same computation gives that 
 \eqref{C4} implies \eqref{wconv}. 
 
Thus Theorem \ref{teoHJB} and Theorem \ref{convex} give the conclusion.
\end{proof}
\begin{rem} \upshape  If $c$ is bounded away from 0 for $|x|$ large, condition \eqref{C4} is satisfied if $a = o(|x|^2)$ and $b=o(|x|)$ as $x\to\infty$. If $b$ is bounded and $a = o(|x|)$, $c$ can vanish as $x\to\infty$ provided $c(x,\a)\geq c_o/|x|$ with $c_o>0$.
\end{rem}
\begin{rem} 
\label{rem:OU}
\upshape 
The condition 
\begin{equation}
\label{C4bis}
\limsup_{|x|\to\infty} \sup_{\a\in A} (tr\, a(x,\a) + b(x,\a)\cdot x) < 0 
\end{equation}
is sufficient for  \eqref{C4} in view of  \eqref{C2}; it means that the vector field $b$ points toward the origin for $|x|$ large enough and all $\a$, and its inward component is large enough compared to the diffusion matrix $a$. It is satisfied if $a = o(|x|^2)$ and the drift $b$ is a controlled perturbation of a mean reverting drift of Ornstein-Uhlenbeck type, that is, for some $m\in\R^N, \gamma>0$,
\begin{equation}
\label{OU}
b(x,\a)=\gamma(m-x) +\tilde b(x,\a) , \qquad \lim_{x\to\infty} \sup_{\a\in A}\frac{\tilde b(x,\a)\cdot x}{|x|^2}=0 .
\end{equation}
More generally, \eqref{C4bis} holds if there exist $\delta\geq 0, \gamma >0$, and $a_o<\gamma$ such that
\[
\sup_{\a\in A}  b(x,\a)\cdot x = -\gamma |x|^\delta + o(|x|^\delta) , \quad \sup_{\a\in A} tr\, a(x,\a) \leq a_o|x|^\delta + o(|x|^\delta) , 
\]
as $|x|\to\infty$.
\end{rem}
\subsection{Pucci extremal operators}\label{sec:pucci} 
Important examples of uniformly elliptic HJB operators are the Pucci extremal operators. In particular, the minimal operator $\mathcal M^-$ has the form \eqref{HJB}, i.e., it is concave in $u$.
Fix $0< \lambda\leq \Lambda$ and denote with  $S^N$ the set of   $N\times N$ symmetric matrices.  For any $X\in S^N$ define 
\begin{equation}
\mathcal M^-(X):=\inf\{-\mbox{tr} (MX): M\in S^N,\lambda I\leq M \leq \Lambda I
\}\label{pucci-}
\end{equation}
 The operator can also be written as 
$$
\mathcal M^-(X)= -\Lambda \sum_{e_i>0} e_i- \lambda \sum_{e_i<0} e_i,
$$
where $e_i$ are the eigenvalues of $X$. 

The Pucci maximal operator  $\mathcal M^+$ is defined as
\begin{multline}
\mathcal M^+(X):=\sup\{-\mbox{tr} (MX): M\in S^N,\lambda I\leq M \leq \Lambda I
\}
= -\lambda \sum_{e_i>0} e_i- \Lambda \sum_{e_i<0} e_i.\label{pucci+}
\end{multline}

Next we prove the Liouville property for subsolutions of the equation
\begin{equation}
\mathcal M^-(D^2u) + H(x,u, Du)=0 , \qquad\text{in } \R^N ,
\label{pucciHJ}
\end{equation}
where
\begin{equation}
H(x,t, p):= \inf_{\a\in A}\{c(x,\a) t
- b(x,\a)\cdot p\} ,
\label{Hami}
\end{equation}
and supersolutions to   \begin{equation}
\mathcal M^+(D^2u) + \tilde H(x,u, Du)=0 , \qquad\text{in } \R^N ,
\label{pucci+HJ}
\end{equation}
where
\begin{equation*}
\tilde H(x,t, p):= \sup_{\a\in A}\{c(x,\a) t
- b(x,\a)\cdot p\}.
\end{equation*}
In both cases we assume that  the data $b, c$ satisfy conditions \eqref{C1} and  \eqref{C2}. 

For these operators the condition \eqref{C4} reads
\begin{equation*}
 \sup_{\a\in A} ( b(x,\a)\cdot x - c(x,\a)|x|^2/2) \leq  - N\Lambda  \quad\text{ for } |x|\geq R_o .
\end{equation*}
The next result improves a bit Corollary \ref{cor1} in this case, for sub- and supersolutions with growth at infinity controlled by  $\log|x|$.
\begin{cor}
\label{cor2}
Under the previous conditions on $H$ assume
\begin{equation}\label{C5}
\sup_{\a\in A} ( b(x,\a)\cdot x - c(x,\a)|x|^2\log(|x|)) \leq  \lambda
 - (N-1)\Lambda  \quad\text{ for } |x|\geq R_o .
\end{equation}
\begin{itemize}
 \item Let $u\in USC(\R^N)$ be a  subsolution of \eqref{pucciHJ}
such that $\limsup_{|x|\to +\infty} \frac{u(x)}{\log|x|} \leq 0$. Assume that either $c(x,\a)\equiv 0$ or $u\geq 0$, then
$u$ is a constant.
\item Let $v\in LSC(\R^N)$ be a  supersolution of \eqref{pucci+HJ}
such that $\liminf_{|x|\to +\infty} \frac{v(x)}{\log|x|} \geq 0$. Assume that either $c(x,\a)\equiv 0$ or $v\leq 0$, then
$v$ is a constant.\end{itemize} 
\end{cor}
\begin{proof} We have to check the property \eqref{w} (and \eqref{wconv}) and we choose $w(x) =\log(|x|)$ (resp. $W(x)=-\log |x|$). 
We recall that for a function of the form $w(x) =\Phi(|x|)$ with $\Phi : (0,+\infty)\to\R$ of class $C^2$ the eigenvalues of $D^2w(x)$, $x\ne 0$, are  $\Phi''(|x|)$, which is simple, and $\Phi'(|x|)/|x|$ with multiplicity $N-1$, see Lemma 3.1 of \cite{CL}. Since $\log$ is increasing and concave we get
\[
M^-(D^2w)=-\Lambda \frac{N-1}{|x|^2} + \frac{\lambda}{|x|^2} \qquad M^+(D^2 W)=\Lambda \frac{N-1}{|x|^2} - \frac{\lambda}{|x|^2} . 
\]
Thus $w$ is a supersolution of \eqref{pucciHJ} at all points where
\[
 \frac{\lambda}{|x|^2}-\Lambda \frac{N-1}{|x|^2} + \inf_{\a\in A}\left\{c(x,\a) \log(|x|) - \frac{b(x,\a)\cdot x}{ |x|^2}\right\}\geq 0 ,
 \]
and this inequality holds for $|x|\geq R_o$ 
 under condition \eqref{C5}. Analogously one checks that condition \eqref{C5} implies that $W$ is a subsolution to \eqref{pucci+HJ} for $|x|\geq R_o$. 
Therefore Theorem \ref{teoHJB} and Theorem \ref{convex} give the conclusion.
\end{proof}
\begin{rem}\upshape
A 
condition that implies  \eqref{C5} and therefore the Liouville property is the following
\begin{equation*}\label{C6}
\limsup_{|x|\to\infty} \sup_{\a\in A}  b(x,\a)\cdot x <  \lambda - (N-1)\Lambda ,
\end{equation*}
because $c\geq 0$.
\end{rem}
%

\section{Uniformly elliptic operators}
\label{sect:ue}
In this section we consider fully nonlinear equations of the general form
\begin{equation}
F(x,u, Du, D^2u) =0 , \qquad\text{in } \R^N ,
\label{F}
\end{equation}
where $F :  \R^N\times \R\times  \R^N\times  S^N \to\R$ is continuous and uniformly elliptic, namely, there exist constants $0<\lambda\leq\Lambda$ such that 
\begin{equation}
\lambda tr(Q)\leq F(x,t, p, X) -  F(x,t, p, X+Q)\leq  \Lambda tr(Q),
\label{unell}
\end{equation}
for all $x, p\in\R^N$, $t\in\R$, $X, Q\in S^N$, $Q\geq 0$. Then it is well-known that
\begin{equation*}
\mathcal M^-(X) \leq F(x,t, p, X) -  F(x,t, p, 0)\leq \mathcal M^+(X) .
\end{equation*}
Now assume that 
\begin{equation}
\label{F>H}
  F(x,t, p, 0)\geq H(x,t,p)
\end{equation}
for some concave Hamiltonian of the form 
 \eqref{Hami}. Then 
  Corollary \ref{cor2} immediately gives the following.
\begin{cor}
\label{cor3}
Assume \eqref{unell} and \eqref{F>H} with $H$ given by  \eqref{Hami} and $b, c$ satisfying  \eqref{C1}, \eqref{C2}, and \eqref{C5}.
Let $u\in USC(\R^N)$ be a  subsolution of \eqref{F}
such that $\limsup_{|x|\to +\infty} \frac{u(x)}{\log|x|} \leq 0$. Assume that either $c(x,\a)\equiv 0$ or $u\geq 0$, then
$u$ is a constant.
\end{cor}
\begin{proof}
It is enough to observe that $u$ satisfies
\[
\mathcal M^-(D^2u)+H(x,u, Du)\leq 0 \quad\text{ in } \R^N
\]
and apply Corollary \ref{cor2}.
\end{proof}
The analogous statement holds for supersolutions.
\begin{cor}
\label{cor3primo}
Assume \eqref{unell} and \begin{equation*} 
  F(x,t, p, 0)\leq \tilde H(x,t, p):= \sup_{\a\in A}\{c(x,\a) t
- b(x,\a)\cdot p\} ,
\end{equation*} 
where $b,c$ satisfies  and $b, c$ satisfying  \eqref{C1}, \eqref{C2}, and \eqref{C5}.
Let $v\in LSC(\R^N)$ be a  supersolution of \eqref{F}
such that $\liminf_{|x|\to +\infty} \frac{v(x)}{\log|x|} \geq 0$. Assume that either $c(x,\a)\equiv 0$ or $v\leq 0$, then
$v$ is a constant.
\end{cor}
We specialize the last corollaries to a class of examples that are useful for comparing with the existing literature.
Assume 
\begin{equation}
\label{F>g|p|}
  F(x,t, p, 0)\geq - \ol b(x)\cdot p-g(x)|p|+\ol c(x) t
\end{equation}
where $\ol b : \R^N\to \R^N$ and $g : \R^N\to \R$ are locally Lipschitz, $\ol c : \R^N\to \R$ is continuous,  and
 \begin{equation}
\label{segno}
g\geq 0 , \qquad \ol c\geq 0 .
\end{equation}
\begin{cor}
\label{cor4}
Assume \eqref{unell}, \eqref{F>g|p|}, \eqref{segno}, and
\begin{equation}\label{C6.5}
\ol b(x)\cdot x +g(x)|x|  \leq  \ol c(x)|x|^2\log(|x|)+ \lambda
 - (N-1)\Lambda  \quad\text{ for } |x|\geq R_o .
\end{equation}
Let $u\in USC(\R^N)$ be a  subsolution of \eqref{F}
such that $\limsup_{|x|\to +\infty} \frac{u(x)}{\log|x|} \leq 0$. Assume that either $\ol c\equiv 0$ or $u\geq 0$, then
$u$ is a constant.
\end{cor}
\begin{proof}
We observe that $-|p|=\min_{|\a| =1} [-\a\cdot p]$, so the right hand side of \eqref{F>g|p|} can be written in the form of a concave Hamiltonian \eqref{Hami}  with $b(x,\a)= \ol b(x)+g(x)\a$, $A=\{\a\in\R^N : |\a| =1\}$, and $c(x,\a)=\ol c(x)$. Then condition \eqref{C5} becomes \eqref{C6.5}, \eqref{C1} and \eqref{C2} are satisfied,  and the conclusion follows from Corollary \ref{cor3}.
\end{proof}

Analogously we get the result for supersolutions. 
\begin{cor}
\label{cor4bis}
Assume \eqref{unell},   \[F(x,t, p, 0)\leq  -\ol b(x)\cdot p+g(x)|p|+\ol c(x) t
\]
where $\ol b, g, \ol c$ are as in Corollary \ref{cor4} and satisfy \eqref{C6.5}.
Let $v\in LSC(\R^N)$ be a  supersolution of \eqref{F}
such that $\liminf_{|x|\to +\infty} \frac{v(x)}{\log|x|} \geq 0$. Assume that either $\ol c\equiv 0$ or $v\leq 0$, then
$v$ is a constant.
\end{cor}
\begin{rem}
\label{rem:CDC}
\upshape

A  set of sufficient conditions for  \eqref{C6.5}, containing the case that  $\ol b$ is the drift of an Ornstein-Uhlenbeck process, is the following:
there exist $\delta\geq 0, \gamma >0$, such that
\[
\ol  b(x)\cdot x = -\gamma |x|^\delta + o(|x|^\delta) , \quad g(x) = o(|x|^{\delta-1}) , \text{ as } |x|\to\infty .
\] 

Finally,  \eqref{C6.5} holds also if, instead,
\[
\liminf_{|x|\to\infty} \ol c(x)>0 , \quad |\ol b(x)|, g(x) = o(|x|\log (|x|)) , \text{ as } |x|\to\infty .
\]
\end{rem}
\section{Quasilinear hypoelliptic operators}
\label{sect:hyp}
In this section we consider equations of the form
\begin{equation}
\label{QL}
-tr (a(x) D^2u) + \inf_{\a\in A} \{-b(x,\a)\cdot Du + c(x,\a) u\}=0 , \qquad\text{ in } \R^N , 
\end{equation}
\begin{equation}
\label{QL+}
-tr (a(x) D^2u) + \sup_{\a\in A} \{-b(x,\a)\cdot Du +c(x,\a) u\}=0 , \qquad\text{ in } \R^N , 
\end{equation}
where $a(x)=\sigma(x)\sigma^T(x)$ for some locally Lipschitz $N\times m$ matrix $\sigma=(\sigma_{ij})$ and the coefficients $b, c$ satisfy the conditions \eqref{C1} and \eqref{C2}. 

Instead of 
the uniform ellipticity \eqref{C3} we assume first
\begin{equation}
\label{nondeg}
\forall\, R>0 \quad \text{either } \inf_{|x|\leq R, \a\in A}c(x,\a)>0 \quad 
\text { or } \quad\exists i :  \inf_{|x|\leq R} \sum_{j=1}^m \sigma_{ij}^2(x)>0 ,
\end{equation}
which will ensure the Comparison Principle on bounded sets. 

Sufficient conditions for the Strong Maximum Principle can be given by means of \emph{subunit vector fields} $\tau$ for the matrix $a$, namely, $\tau : \R^N\to\R^N$ such that $\xi^T a(x) \xi\geq (\tau(x)\cdot\xi)^2$ for all $\xi\in\R^N$. Of course each column of $\sigma$ is a subunit vector field, but also $\eta a_j$, where $a_j$ is the $j$-th column of the matrix $a$ and $\eta>0$ is small enough, see, e.g., \cite{BDL}. The second assumption will be
\begin{
multline}
\label{Hor}
\qquad\text{ there exist subunit vector fields } \tau_j ,\; j=1,\dots,n, \text{ of class } C^\infty
\\ \text{and generating a Lie algebra of full rank $N$ at each point } x\in\R^N. \end{
multline}
This classical condition of H\"ormander can be weakened: see Remark \ref{rem:Hor} after the next result.
\begin {cor}\label{corQL}
Let the previous assumptions and
\begin{equation}
\label{C4QL} 
|\sigma(x)|^2 +\sup_{\a\in A} (b(x,\a)\cdot x - c(x,\a)|x|^2/2) \leq 0 \quad\text{ for } |x|\geq R_o ,
\end{equation} hold. 
\begin{itemize}
\item  Let $u\in USC(\R^N)$ be a  subsolution of \eqref{QL}
such that $\limsup_{|x|\to +\infty} \frac{u(x)}{|x|^2} \leq 0$. 
If either $c(x,\a)\equiv 0$ or $u\geq 0$, then
$u$ is a constant.
\item  Let $v\in LSC(\R^N)$ be a  supersolution of \eqref{QL+}
such that $\liminf_{|x|\to +\infty} \frac{v(x)}{|x|^2} \geq 0$. 
If either $c(x,\a)\equiv 0$ or $v\leq 0$, then
$v$ is a constant.
\end{itemize} 
\end{cor}
\begin{proof}
The Comparison Principle in bounded sets under the first condition in \eqref{nondeg} is standard \cite{CIL}, 
whereas under the second condition it is Corollary 4.1 in  \cite{BM}.

The assumption \eqref{Hor} implies the Strong Maximum Principle and the Strong Minimum Principle for both equations \eqref{QL} and \eqref{QL+} by the results of \cite{BDL}.

Finally, it is easy to see that $tr\, a(x) = |\sigma(x)|^2$, where $|\sigma|$ denotes the Euclidean norm of the matrix $\sigma$. Then \eqref{C4QL} is equivalent to  \eqref{C4} and so $w(x) =|x|^2/2$ is a supersolution of \eqref{QL} as in the  proof of
 Corollary \ref{cor1}, whereas $W(x)=-|x|^2/2$ is a subsolution to \eqref{QL+}. 
Thus Theorem \ref{teoHJB} and Theorem \ref{convex}
  give the conclusions.
\end{proof}
\begin{rem}
\upshape
The second condition in  \eqref{nondeg} is satisfied in many classical examples of subelliptic operators, e.g., if the columns of the matrix $\sigma$ are the  generators of a Carnot group, see \cite{BLU}. It can be further relaxed to $\inf_{|x|\leq R} \sum_{i=1}^N \sum_{j=1}^m \sigma_{ij}^2(x)>0$ provided that there exist vector fields $\tilde b : \R^N\times A\to \R^m$ such that $b(x,\a)=\sigma(x) \tilde b(x,\a)$, see Corollary 4.1 in  \cite{BM}.
\end{rem}
\begin{rem}
\label{rem:Hor} \upshape
The general sufficient condition for the Strong Maximum Principle originating in Bony's work and extended to nonlinear operators in \cite{BDL} is the following.
Suppose there exist Lipschitz continuous subunit vector fields $\tau_j ,\; j=1,\dots,n$, and consider the control system
\[
\dot y(t)=\sum_{j=1}^n \beta_j(t) \tau_j(y(t)) ,
\]
where the control $\beta_j$ take values in a compact neighborhood $B$ of the origin. Assume that each $x\in\R^N$ has a neighborhhod such that all points can be reached by a trajectory of the system starting at $x$, i.e., there exists $r>0$ such that for all $z$ with $|z-x|<r$ there are measurable $\beta_j : [0,+\infty)\to B$ for which the solution of the system with $y(0)=x$ satisfies $y(t)=z$ for some $t>0$. Then the strong maximum and minimum principles hold for \eqref{QL} and \eqref{QL+}. The H\"ormander condition \eqref{Hor} is sufficient for this reachability property but not necessary. In particular, the smoothness of the vector fields can be relaxed.
\end{rem}
\begin{rem}
\upshape
Fully nonlinear HJB equations involving hypoelliptic operators $L^\a$ can also be considered. Sufficient conditions for the Strong Maximum Principle are given in 
\cite{BDL}, but they are not as explicit as \eqref{Hor} or the condition described in the preceding Remark \ref{rem:Hor}. They still concern a reachability property of a control system, but instead of a deterministic one it is either a 
diffusion process or a deterministic differential game, and therefore the formulation of such conditions is more technical.
\end{rem}
\section{Large-time stabilization in parabolic equations}
\label{sect:lts}
We consider   the operators with continuous coefficients \eqref{HJB} and  \eqref{HJB2} 
introduced in Section \ref{sect:HJB}.
For   functions $u : [0, +\infty)\times \R^N \to \R$ we denote with $Du=D_x u$ and $D^2u=D^2_x u$ the first and second partial derivatives of $u$ with respect to the space variables.
\begin{cor} 
\label{cor:evol}
Assume $G$ satisfies the conditions \eqref{CP},  \eqref{SMP},  \eqref{w}. If  $u\in USC([0, +\infty)\times \R^N)$ satisfies 
$$
u_t + G[u]\leq 0    \qquad \text{ in }  \R^N\times (0, +\infty),
$$ 
\[\limsup_{|x|\to +\infty} \frac{u(t,x)}{w(x)} \leq 0 \qquad \text{ uniformly in }t\in [0, +\infty),\]
and either $c\equiv 0$ or $u\geq 0$,
then 
$$
\limsup_{t\to +\infty, y\to x} u(t,y) =:\overline u(x)
$$
 is constant with respect to $x$.			
\end{cor}
\begin{proof}
Consider the rescaled function $v_\eta(t,x):=u(t/\eta,x)$ and note that it is a subsolution of 
$$
\eta \frac{\de v_\eta}{\de t} + G[v_\eta]\leq 0    \qquad \text{ in }  \R^N\times (0, +\infty) .
$$ 
By the stability of viscosity subsolutions, the function
$$
\overline v(t,x) := \limsup_{\eta\to 0, s\to t, y\to x} v_\eta(s,y) 
$$
is a subsolution of $ G[\overline v]\leq 0 $ 
in $  \R^N\times (0, +\infty)$. 
On the other hand, by the very definitions, $\overline v(t,x)= \overline u(x)$ for any $t>0$. Moreover, it is easy to check that 
$\limsup_{|x|\to +\infty} \frac{\overline{ u}(x)}{w(x)} \leq 0$ and that, if $u\geq 0$,   also $\overline u\geq 0$. 
Then
\[
 G[\overline u]\leq 0  \qquad \text{ in }  \R^N ,
\]
and we can use Theorem \ref{teoHJB} to conclude that $\overline u$ is a constant.
\end{proof}

\begin{rem}
\label{rem:evol} 
\upshape It is immediate to prove the analogous for supersolution. Assume $\tilde G$ satisfies \eqref{CP}, $(2')$ and \eqref{wconv}. 
Let $v$ be a  LSC supersolutions of $
u_t + \tilde G[u]\geq 0$,  such that $\liminf_{|x|\to +\infty} \frac{v(t,x)}{W(x)} \leq 0$ uniformly in $t$. Assume moreover that either $c\equiv 0$ or $v\leq 0$. Then 
 $\liminf_{t\to +\infty, y\to x} u(t,y)$ is a constant. \end{rem} 
Now we consider the Cauchy problem
\begin{equation}\label{Cauchy}
\left\{
\begin{array}{ll}
u_t + F(x,
 Du, D^2u)=0 \quad\ & \text{ in  }(0, +\infty)\times \R^N \\
u(0,x) = h(x)  \quad &   \text{ in  }\R^n ,
\end{array}
\right.\,
\end{equation}
where $F$ is uniformly elliptic and satisfies 
\begin{equation}
\label{Fg|p|}
 - b_1(x)\cdot p-g_1(x)|p| 
  \leq  F(x,
   p, 0)\leq  - b_2(x)\cdot p +g_2(x)|p| 
\end{equation}
with $b_i : \R^N\to \R^N$ and $g_i : \R^N\to \R$ bounded and locally Lipschitz, $i=1,2$. 
%
\begin {thm}
\label{longt}
Assume $F$ satisfies the structural conditions for the comparison principle bewteen a sub- and a supersolution of \eqref{Cauchy}, as well as \eqref{unell},  \eqref{Fg|p|} with $g_i\geq 0$  and 
\begin{equation}\label{C10}
b_i(x)\cdot x +g_i(x)|x|  \leq  
\lambda - (N-1)\Lambda  \quad\text{ for } |x|\geq R_o , i=1,2 .
\end{equation}
Suppose also that $h \in BUC(\R^N)$. 
Then there exist a unique solution  $u$  of  \eqref{Cauchy} H\"older continuous in 
$[0, +\infty)\times \R^N$ and constants $\ol u, \ul u \in\R$ such that
\begin{equation}
\label{lt_limits}
\limsup_{t\to +\infty} u(t,x) = \overline u , \quad \liminf_{t\to +\infty} u(t,x) =\ul u , \quad \text{ for all } x\in \R^N.
\end{equation}
In particular, if for some $\ol x$ the limit $\lim_{t\to +\infty} u(t,\ol x)$ exists, then $\lim_{t\to +\infty} u(t, x)$ exists for all $x$, it is independent of $x$, and locally uniform.
\end{thm}
\begin{proof}
We divide the proof in three steps. 

{\bf Step 1} . We show that  \[\limsup_{t\to +\infty, y\to x} u(t,y)=\overline u(x),\qquad\liminf_{t\to +\infty, y\to x} u(t,y)=\underline{u}(x)\] are constants, that is $\ol u(x)\equiv \ol u$ and $\underline u(x)\equiv\underline u$ for every $x$. 

The existence and uniqueness of a solution $u \in BUC([0, +\infty)\times \R^N)$ for all $T>0$ follows from the comparison principle and Perron's method by standard theory. We must prove global regularity estimates.
We will use several times that, by \eqref{unell} and  \eqref{Fg|p|},
\begin{equation}
\label{bounds_on_F}
\mathcal M^-(X) - b_1(x)\cdot p-g_1(x)|p|   \leq  F(x, p, X)\leq  \mathcal M^+(X)- b_2(x)\cdot p +g_2(x)|p| .
\end{equation}
First observe that it implies that any constant solves the PDE; consequently, we
have the bounds
$\inf h\leq u(t,x) \leq \sup h$, for every $x,t$,  by the Comparison Principle.

Arguing as in  Corollary \ref{cor:evol}. we get that $\limsup_{t\to +\infty, y\to x} u(t,y)$ is  a subsolution to
\[\mathcal{M}^- (D^2u)-b_1(x)\cdot Du-g_1(x)|Du|\leq 0 \] and  $\liminf_{t\to +\infty, y\to x} u(t,y)$  is a supersolution to \[\mathcal{M}^+(D^2u)-b_2(x)\cdot Du+g_2(x)|Du|\geq 0.\]
Note that condition \eqref{C10}  coincides with \eqref{C6.5}, then we can apply  Corollary \ref{cor4} and Corollary \ref{cor4bis} and  conclude that 
\[
\limsup_{t\to +\infty, y\to x} u(t,y)=\overline u(x),\qquad\liminf_{t\to +\infty, y\to x} u(t,y)=\underline{u}(x)
\]
 are constants, that is $\ol u(x)\equiv \ol u$ and $\underline u(x)\equiv\underline u$
for every $x$. 

{\bf Step 2}. We show that if $h$ is smooth  with bounded first and second derivatives, then the conclusion holds. 

We apply the theory of uniformly elliptic equations for $t$ fixed.
 From the comparison principle we get the estimate
$$
|u(t,x)-h(x)|\le
Ct 
\qquad\text{ on } [0, +\infty)\times \R^N,
$$
 for the constant
$C:=\sup_{x}|F(x,D h,D^2 h)|.$
  By applying again the
comparison principle we obtain 
$$|u(t+s,x)-u(t,x)|\le
\sup_{x\in\R^N}|u(s,x)-h(x)|\le Cs \qquad\text{ on }
[0,+\infty)\times\R^N
$$
 for all  $s>0$.
In particular, we have $|
u_t|\le C$ in the viscosity sense. From this, \eqref{bounds_on_F}, and the boundedness of $b_i, g_i$, it is easy to deduce that
the partial function $u(t,\cdot)$ satisfies  for all $t>0$
\begin{equation}\label{subsol}
\mathcal M^-(D^2u)
- C_1|{Du}|
 \leq C , \quad \mbox{in }\R^N ,
\end{equation}
\begin{equation}\label{supersol}
\mathcal M^+(D^2u)+ C_1|{Du}| \geq -C , \quad \mbox{in }\R^N 
\end{equation}
 in the viscosity sense.
Then we can apply the  estimates of Krylov-Safonov type as stated 
in Thm. 5.1 of \cite{Tru88}. 
By \eqref{subsol} $u(t,\cdot)$ satisfies a local maximum principle with constants depending only on $N, \lambda, \Lambda, C, C_1,$ and  $\|h\|_\infty$, whereas by \eqref{supersol} $u(t,\cdot)$ satisfies a weak Harnack inequality with constants depending only on the same quantities.
The combination of these two estimates with the classical Moser iteration technique 
implies that the family $u(t,\cdot)$ is equi-H\"older continuous. 
Since $u$ is Lipschitz continuous in $t$, we conclude that it is H\"older 
 continuous in $[0, +\infty)\times \R^N$. This implies that $\limsup_{t\to +\infty} u(t,x)=\limsup_{t\to +\infty, y\to x} u(t,y) $ and
 $\liminf_{t\to +\infty} u(t,x)=\liminf_{t\to +\infty, y\to x} u(t,y)$.
 So, by Step 1, $\limsup_{t\to +\infty} u(t,x)=\overline u$, and $\liminf_{t\to +\infty} u(t,x)=\underline{u}$ for every $x\in\R^N$. 
 
%
%

{\bf Step 3}. We conclude  for general $h\in BUC(\R^N)$. 

We mollify $h$ and take a sequence of smooth functions  $(h_k)$ with bounded first and second derivatives
 converging uniformly to $h$.  The comparison principle implies
 that the associated sequence of solutions $(u_k)$ converges uniformly to $u$
 on $[0, +\infty)\times \R^N$.  
 Moreover, for each fixed $k$, we proved in Step 2  that  $\limsup_{t\to +\infty} u_k(t,x) = \overline u_k$ and $\liminf_{t\to +\infty} u_k(t,x) = \underline u_k$. 
 Note that both $\overline u_k$ and $\underline u_k$ are bounded (due to the fact that $(h_k)$ are uniformly bounded), so we can extract a converging subsequence. 
 
Let $t_n\to +\infty$ and $x_n\to x$ such that $\lim _n u(t_n,x_n)=\overline{u}$. 
By uniform convergence of  $u_k$ to $u$ in $[0, +\infty)\times \R^N$, for every $\eps>0$ there exists $\overline k$ such that
for every $k\geq \overline k$, 
\[u_k(t_n,x_n)-\eps\leq u(t_n,x_n)\leq u_k(t_n,x_n)+\eps \qquad\forall k\geq \overline k.\] Letting $n\to +\infty$ we obtain from the previous inequalities 
\[\overline u\leq \overline u_k+\eps \qquad\forall k\geq \overline k,\]  and then letting $k\to +\infty$, we conclude
\[\overline u\leq \lim_k \overline u_k.\]

Let fix $x$, $k\geq \overline k$ and $t_n^k\to +\infty$ such that $\lim_n u_k(t_n^k, x)=\overline u_k$. Then there exists $n_k$ such that for every $n\geq n_k$
\[\overline u_k\leq u_k(t_n^k, x)+\eps\leq u(t_n^k, x)+2\eps.\]
Letting $n\to +\infty$, we get that \[\overline u_k \leq \limsup_{t\to +\infty} u(t,x)+2\eps \leq \limsup_{t\to +\infty, y\to x} u(t,y)+2\eps=\overline u+2\eps.\] 
So, letting $k\to +\infty$, we get that $\lim_k \overline u_k \leq \limsup_{t\to +\infty} u(t,x)\leq \overline u$.
Therefore, we conclude  that $\overline u=\limsup_{t\to +\infty} u(t,x)=\lim_k\overline u_k$. 

The same argument gives the result for the $\liminf$. 
\end{proof}
\begin{rem}
\label{stab:linear}
\upshape
An example where the last statement of Theorem \ref{longt} holds true is a linear operator 
 whose coefficients satisfy, for some $R_o, M_o>0$,
 \begin{equation}
 \label{dissip}   
tr\, a(x)+b(x)\cdot x\leq -M_o \qquad \forall\, |x|>R_o ,
\end{equation}
which is equivalent to \eqref{C4bis} and slightly stronger than \eqref{C4}.
Then the stochastic process $d X_t=b(X_t) dt + \sqrt 2 \sigma (X_t) dW_t$ generated by the operator $L=tr\, \sigma\sigma^TD^2 + b\cdot D$ is ergodic with a unique invariant probability measure $\mu$, see, e.g., \cite{BCM}. 
Moreover $\lim_{t\to +\infty}\mathbb{E}h(X_t) =\int_{\R^N} h(y) d\mu(y)$ locally uniformly in $x=X_0$ (Prop. 4.4 of \cite{BCM}).
Since the solution of the Cauchy problem \eqref{Cauchy} is 
$u(t,x)=\mathbb{E} h(X_t)$, 
we have   in this case that 
 \begin{equation}
 \label{average}   
 \ol u = \ul u =\int_{\R^N} h(y) d\mu(y).
 \end{equation}
\end{rem}
\begin{rem}
\label{stab:calore}
\upshape
Without a dissipativity condition like \eqref{dissip} the equality $ \ol u = \ul u$ 
  cannot be true for all 
   bounded initial data $h$, even for the heat equation in dimension $N=1$, see the example in \cite{CE}. For linear equations various authors 
  studied the further averaging properties of $h$ 
  necessary and sufficient for the stabilization to a constant, $ \ol u = \ul u$, see \cite{EKT} and the references therein.
\end{rem}
\begin{rem}
\label{stab:nonlin}
\upshape
For a nonlinear operator $F$ of HJB type one may hope for a 
 formula like \eqref{average}   if an associated optimal control problem  with long-time 
 cost or payoff (a so-called ergodic control problem) has an optimal feedback producing an ergodic process with unique invariant measure $\mu$. In principle such a feedback can be synthesized from a stationary HJB equation in $\R^N$ (see next section).  So far, this has been done with PDE methods  only in some special model problems of the form  $F[u]=-\Delta u +|Du|^q+l(x)$ with $q>1$, see, e.g., \cite{Ichi, Cir}. Representation formulas like \eqref{average} have been obtained by probabilistic methods  under appropriate dissipativity conditions on the control system in \cite{ABG}, see also \cite{CFP}. 
 Results of this kind under our growth assumption \eqref{Fg|p|} look considerably harder and are beyond the scope of this paper.
 
 For general operators $F$ of HJB type with all the data $\Z^N$-periodic one can exploit the compactness of the flat torus to show that $\ol u = \ul u$, although an integral representation as  \eqref{average} of such constant is not available. See \cite{BA2}, where the operators can also be of Isaacs type, i.e., inf sup or sup inf of linear operators.  Related results for Ornstein-Uhlenbeck type operators of the form $F[u]=-\Delta u +\alpha x\cdot Du +H(Du)+l(x)$ has been obtained in \cite{fil}.  
\end{rem}

\section{Ergodic HJB equations 
 in $\R^N$}
 \label{sect:erg}
In this section we consider the  so-called ergodic HJB equation
 \begin{equation}
 \label{cell}   
F(x, D\chi(x),D^2\chi(x))= c,\qquad   x\in\R^N,
\end{equation}
where the unknowns are $(c,\chi)\in\R\times C(\R^N)$, $F$ is of the form  
 \begin{equation}
 \label{fnuova} 
 F(x, p, X)= \begin{cases} \ \inf_{\alpha\in A} \{-tr\, a(x,\a)X - b(x,\a)\cdot p -l(x,\a)\} \quad\forall\, x, p, X,\\ \text{ or} \\ \ \sup_{\alpha\in A} \{-tr\, a(x,\a)X - b(x,\a)\cdot p -l(x,\a)\} \quad\forall\, x, p, X, \end{cases}
 \end{equation}
the coefficients  $b, a$ satisfy assumptions \eqref{C1} and \eqref{C3}, and   the function $l:\R^N\times A\to \R$ is continuous, bounded, 
and uniformly continuous in $x$, uniformly with respect to $\a$. 

In order to study the well posedness of \eqref{cell}, we need to strengthen  assumption \eqref{w},
by imposing, roughly speaking, that $G[w]\to +\infty$ as $x\to +\infty$, see assumption \eqref{C10strong} below. 

\begin{thm}
\label{cellteo}  
Assume that $F$  is as in \eqref{fnuova}, and that  for every $M>0$ there exists $R>0$ such that 
\begin{equation}
\label{C10strong}
\sup_{a\in A}\{ tr\, a(x,\a) + b(x,\a)\cdot x  \}  \leq -M  \quad\text{ for } |x|\geq R.
\end{equation}
Then exists a unique constant $c\in\R$ for which  \eqref{cell} 
admits a viscosity solution $\chi$ such that 
\begin{equation}
\label{growthcorr} 
\lim_{|x|\to +\infty} \frac{\chi(x)}{|x|^2} =0.
\end{equation}
Moreover $\chi\in C^2(\R^N)$ and is unique up to additive constants among all solutions $v$ to  \eqref{cell} which satisfy \eqref{growthcorr}. 

Finally, if $a(x,\a)$  is 
bounded in $\R^N\times A$, 
then $\chi$ is unique up to additive constants
also among  all solutions $v$ to  \eqref{cell} with polynomial 
 growth at infinity, that is, for which there exists $k\geq 2$ such that 
\[\lim_{|x|\to +\infty} \frac{v(x)}{|x|^k}=0.\]
\end{thm} 
\begin{proof}
The proof is divided in several steps. For a similar construction in bounded domains with irrelevant boundary 
we refer to \cite{BCR} (see also \cite{CCR}), whereas the periodic case is considered in  \cite{AL} and \cite{BA2}. We assume that $F(x,p, X)=\inf_{\a\in A} \{-tr\, a(x, \a)X-b(x, \a)\cdot p-l(x,\a)\}$ (the other case can be treated analogoulsy). 

{\bf Step 1. For every $h\in (0,1]$ there exists $R_h$ such that 
\begin{equation}
\label{c3} 
 -h\frac{|x|^2}{2} +\min_{|x|\leq R_h} u_\delta \leq u_\delta(x)\leq \max_{|x|\leq R_h} u_\delta+h\frac{|x|^2}{2},
\end{equation} 
where $u_\delta$ is a 
bounded solution to 
\begin{equation}
\label{delta}
 \delta u+F(x, Du, D^2 u)=0\qquad \text{ in }\R^N, \ \ \delta>0.
 \end{equation} }

For any $\delta>0$ 
consider the value function of a discounted, infinite horizon, stochastic control problem
\[
u_\d(x):=\sup_{\a_.\in \mathcal{A}} \mathbb E \left[\int_0^{+\infty} e^{-\d t}l(X_t,\a _t) dt\right] ,
\]
where $X_t$ solves
\[
dX_t=b(X_t,\a _t) dt + \sigma (X_t,\a _t) dW_t , \qquad  X_0=x ,
\]
$W_t$ is an $m-$dimensional Brownian motion, $ \mathbb E$ is the expectation, and $\mathcal{A}$ denotes the set of admissible controls (i.e., $\a_. : [0,+\infty)\to A$ progressively measurable with respect to the filtration associated to $W_.$). It is easy to deduce form the definition that
\begin{equation}
\label{b1}
 \|u_\delta\|_\infty\leq \frac{1}{\delta}\|l\|_\infty.
 \end{equation} 
 Moreover it is known that under the current assumptions $u_\d$ is continuous and solves \eqref{delta}, see, e.g., \cite{FS}.

Consider $w(x)=|x|^2/2$. Then,  we get  \begin{equation}\label{c1} 
F(x, Dw, D^2 w)\geq -\sup_{a\in A}\{ tr\, a(x,\a) + b(x,\a)\cdot x  \}-\|l\|_\infty.\end{equation}

Fix $h\in (0,1]$ and let $M=\frac{2}{h} \|l\|_\infty+1$. Choose $R_h>0$ such that \eqref{C10strong} holds for such $M$. 
Then, the function
 \[
 V(x)=h\frac{|x|^2}{2}+\max_{|y|\leq R_h} u_\d
 \] 
 is a supersolution to \eqref{delta} in $|x|>R_h$, due to \eqref{c1} and \eqref{b1}. Indeed 
\begin{eqnarray}\nonumber \delta V+F(x, DV, D^2V) &\geq &\delta h\frac{|x|^2}{2} +\delta\max_{|y|\leq R_h} u_\d -h \sup_{a\in A}\{ tr\, a(x,\a) + b(x,\a)\cdot x  \}-\|l\|_\infty\\ 
&\geq & -\|l\|_\infty +Mh-\|l\|_\infty\geq h>0.\label{c2}\end{eqnarray} 

Note that  $(u_\d -V)(x)\leq 0$ for every $x$ with $|x|\leq R_h$, and $\lim_{|x|\to +\infty} u_\d (x)-V(x)=-\infty$. We claim that $u_\d(x) -V(x)\leq 0$ for every $x$. 
If it were not the case, there would exist a point $\bar x$ such that $|\bar x|>R_h$ and $u_\d (\bar x)-V(\bar x)=\max u_\d -V >0$.  But then \eqref{c2} would contradict the fact that $u_\d$ is a viscosity subsolution to \eqref{delta}.

So we get  that for every $h\in (0,1]$ there exists $R_h$ such that  the second inequality in \eqref{c3}  holds,
where the first inequality is obtained analogously by considering $v(x)=- h\frac{|x|^2}{2}+\min_{|y|\leq R_h} u_\d$. 

{\bf Step 2. The functions $v_\delta=u_\delta-u_\delta (0)$ are equibounded in every compact set $K$}.

Assume by contradiction that there exists   $K$ compact such that $(\eps_\delta)^{-1}:=\|v_\delta\|_{L^\infty(K)} \to +\infty$. Up to enlarging $K$ we can suppose  that 
$K\supset  \{x\ |\ |x|\leq R_1\}$ where $R_1$ has been defined in Step 1. 
 
Define $\psi_\delta =\eps_\delta v_\delta$. Then $\|\psi_\delta\|_{L^\infty(K)}=1$ and $\psi_\delta(0)=0$. 
Moreover, by Step 1, we get that if $x\not\in K$, then \[\psi_\delta(x)\leq \frac{\frac{|x|^2}{2} +\max_{|x|\leq R_1} u_\delta-u_\delta (0)}{\|u_\delta-u_\delta(0)\|_{L^\infty(K)}}
\leq 1+\eps_\delta \frac{|x|^2}{2} \] and \[ \psi_\delta(x)\geq -1-\eps_\delta\frac{|x|^2}{2}.\] 

Therefore the sequence $\psi_\d$ is equibounded in every compact subset of $\R^N$. 
Moreover, since $u_\delta$ solves \eqref{delta}, 
$\psi_\delta$ solves in viscosity sense
\[
\delta \psi_\delta +\delta \eps_\delta u_\delta(0)+ \inf_{\alpha\in A} \{-tr\, a(x,\a)D^2\psi_\delta- b(x,\a)\cdot D\psi_\delta -\eps_\delta l(x,\a) \}=0 .
\] 
Since  $l$   and $\delta u_\delta(0)$ are bounded (uniformly in $\delta$), we argue as in Step 2 of the proof of Theorem \ref{longt}, and we  apply the estimates of Krylov-Safonov type as stated in Thm. 5.1 of \cite{Tru88}. In particular, these imply that the family $\psi_\delta$ is equi-H\"older continuous in every compact set of $\R^N$. Using a diagonal procedure, we 
can find a sequence $\psi_\delta$ converging locally uniformly in $\R^N$ to $\psi$. Moreover, by stability of viscosity solutions, $\psi$ solves in viscosity sense
\[\inf_{\alpha\in A} \{-tr\, a(x,\a)D^2\psi - b(x,\a)\cdot D\psi \}\leq 0\qquad \text{and}\qquad \sup_{\alpha\in A} \{-tr\, a(x,\a)D^2\psi - b(x,\a)\cdot D\psi \}\geq 0.\]

Moreover, we know that $\|\psi\|_{L^\infty(K)}=1$ and $|\psi(x)|\leq 1 $ for $x\in \R^N\setminus K$ since $|\psi_\delta|\leq 1+\eps_\delta \frac{|x|^2}{2}$. 
This implies that $\psi$ attains either a global maximum or a global minimum in $K$, so it is constantly equal to $1$ or to $-1$ by the Strong Maximum  or Minimum Principle (\cite{BDL1}, \cite{BDL}). This contradicts the fact that $\psi(0)=0$. 

{\bf Step 3. Construction of $c$ and $\chi$ solutions to \eqref{cell}.} 

Due to \eqref{b1}, up to extracting a subsequence, 
 $\delta u_\delta(0)$ converges to $-c$ as $\delta\to 0$. Moreover, by Step 2, $v_\delta$ are equibounded in every 
compact set of $\R^N$ and are viscosity solutions to \[\delta v_\delta+\delta u_\delta(0)+F(x, Dv_\delta, D^2 v_\delta)=0.\] Using again Krylov-Safonov type estimates, we get that actually $v_\delta$ are equi-H\"older continuous in every compact set of $\R^N$. So, using a diagonal 
 procedure, we can extract a  subsequence $v_\delta$ which converges locally uniformly to $\chi$. Moreover by stability of viscosity solutions $c,\chi$ solve \eqref{cell}. 

{\bf Step 4. Qualitative properites of $\chi$}. 

Note that the estimates \eqref{c3} is independent of $\delta$, so it holds also for $\chi$: for every $h\in (0,1]$ there exists $R_h>0$ such that \[-h\frac{|x|^2}{2}+\min_{|y|\leq R_h} \chi(y)\leq \chi(x)\leq \max_{|y|\leq R_h} \chi(y)+h\frac{|x|^2}{2}.\] This implies in particular that $\chi$ satisfies the growth condition \eqref{growthcorr}. The regularity of $\chi$ comes from 
elliptic standard regularity theory, see \cite{Tru88}. 

{\bf Step 5. Uniqueness of $c$ and of $\chi$ up to  additive constants.} 
Assume that there exist $c_1\leq c_2$ and two solutions $\chi_1, \chi_2$ to \eqref{cell}  with $c=c_1$ and $c=c_2$, respectively,  which satisfy  \eqref{growthcorr}. 
Then  we get 
\begin{eqnarray*}
 & & \inf_{\a\in A} \{-tr \, a(x, \a)(D^2\chi_1-D^2\chi_2)-b(x, \a)\cdot (D\chi_1-D\chi_2)\} \\ &\leq&  
 \inf_{\a\in A} \{-tr \, a(x, \a)D^2\chi_1-b(x, \a)\cdot D\chi_1-l(x,\a)\}\\ &- &\inf_{\a\in A} \{-tr \, a(x, \a)D^2\chi_2-b(x, \a)\cdot D\chi_2-l(x,\a)\}= c_1-c_2\leq 0,
 \end{eqnarray*}
  where all the equalities and inequalities have to be understood  in the viscosity sense. 
So, by Corollary \ref{cor1} applied to $\chi_1-\chi_2$, we get that $\chi_1-\chi_2$ is a constant. This implies in particular that $c_1=c_2$. 

{\bf Step 6. 
Stronger uniqueness for $a$ 
 bounded.} 

Let consider two solutions  $\chi_1, \chi_2$ to \eqref{cell}  with $c=c_1$ and $c=c_2$ respectively  such that there exists $k\geq 2$
with
\[
\lim_{|x|\to +\infty} \frac{\chi_i(x)}{|x|^k} =0.
\] 
As above, $\chi_1-\chi_2$ satisfies $\inf_{\a\in A} \{-L^\a (\chi_1-\chi_2)\} =c_1-c_2\leq 0$. 

Consider $w(x)=|x|^k/k$.  Then 
\[
L^\a w= |x|^{k-2} \left[tr\, a(x,\a)+b(x,\a)\cdot x
+(k-2)\frac{|\sigma(x,\a)\cdot x|^2}{|x|^2}\right].
\]
Since  $a(x,\a)$ is bounded, also $\sigma(x,\a)$ is bounded, and then 
the term $(k-2)\frac{|\sigma(x,\a)\cdot x|^2}{|x|^2}$ is bounded. 
Therefore condition \eqref{C10strong} implies that there exists $R_o$ such that  $L^\a w(x)\leq 0$ for $|x|\geq R_o$. So, by the same argument of the proof of Corollary \ref{cor1} 
applied to the function $|x|^k/k$ instead of $|x|^2/2$, we get that $\chi_1-\chi_2$ is a constant, and then $c_1=c_2$. 
\end{proof} 
\begin{rem}
\upshape \label{remou} 
 If we strengthen condition \eqref{C10strong}, we can get 
  better   estimates on the growth at infinity of  the solution $\chi$ to \eqref{cell}. 

In particular, if we  substitute assumption \eqref{C10strong} with the following: for some $0<\beta<2$, for every $M>0$ there exists $R>0$ such that  
\begin{equation}
\label{C10ou}
\sup_{a\in A}\{ tr\, a(x,\a) + b(x,\a)\cdot x  \} \leq  -M |x|^{2-\beta} \quad\text{ for } |x|\geq R,
\end{equation}
then the same argument of Theorem \ref{cellteo}, with $w(x)=\frac{|x|^\beta}{\beta}$ in place of $\frac{|x|^2}{2}$, gives that the solution $\chi$ to \eqref{cell} has a strictly sub-quadratic growth at infinity, that is 
 \begin{equation}
 \label{subq}
 \lim_{|x|\to +\infty} \frac{\chi(x)}{|x|^\beta}=0.
 \end{equation}
In particular, for perturbations of the Ornstein-Uhlenbeck drift as in \eqref{OU} of Remark \ref{rem:OU} the solution $\chi$ satisfies \eqref{subq} for all $\beta >0$, so it grows at infinity less than any polynomial.

In the limit case where \eqref{C10ou} holds with $\beta=0$
we can use $w(x)=\log |x|$  
and get that the solution $\chi$ to \eqref{cell} has sublogarithmic growth at infinity, that is, 
\[\lim_{|x|\to +\infty} \frac{\chi(x)}{\log |x|}=0.\]  
\end{rem} 
\begin{rem}
\upshape \label{remou2} On the other hand, 
 if we weaken assumption \eqref{C10strong}, we get weaker results on the growth at infinity of $\chi$. For example, let us assume that there exist $k>2$ 
 and $R_o>0$ such that  
 \[\sup_{a\in A}\{ tr\, a(x,\a) +(k-2)\frac{|\sigma(x,\a)\cdot x|^2}{|x|^2}+b(x,\a)\cdot x  \} \leq  0 
   \quad\text{ for } |x|\geq R_o,\  i=1,2. \]
 Then, arguing again as in Theorem \ref{cellteo} with $w(x)=|x|^k/k$, we get that the solution $\chi$ to \eqref{cell} satisfies   \[\lim_{|x|\to +\infty} \frac{\chi(x)}{|x|^k}=0 \]
 instead of \eqref{growthcorr}.
\end{rem} 

We conclude this section with some results on the possible boundedness of  
the solution $\chi$  to the ergodic equation \eqref{cell}. 
The next example shows that in general it can be unbounded.
\begin{ex}\upshape
Consider the case of $N=1$, $A$ a singleton, 
$b(x)=-x$, $a(x)=1$ and $l(x)= 2 \frac{x^4+2x^2-1}{(x^2+1)^2}$. In this case \eqref{C10ou} is satisfied for every $\beta<2$ and the ergodic  problem \eqref{cell} reads as follows
 \[
  -\chi'' +x \chi' - 2\frac{x^4+2x^2-1}{(x^2+1)^2}=c.
  \] 
So, by Theorem \ref{cellteo} and Remark \ref{remou}, there exists a unique $c$ for which this equation has a solution  which satisfies \eqref{subq}. 
It is easy to check that this solution is $c=0$ and $\chi(x)=\log(1+x^2)$ up to addition of constants. 
\end{ex}

On the other hand, the solution $\chi$ to \eqref{cell}  is bounded if we strenghten condition \eqref{C10ou} to the following:
there exist $\rho>0$ 
and $R>0$ such that 
\begin{equation}
\label{C10extrastrong}
\sup_{a\in A}\{ tr\, a(x,\a) + b(x,\a)\cdot x  \}  \leq  
-\frac{2|c|+\|l\|_\infty}{\rho} |x|^{2+\rho}     \quad\text{ for } |x|\geq R,
\end{equation}
where $c$ is the constant solving the ergodic equation \eqref{cell}. This is proved  in the following proposition (see also \cite{CCR} for a similar result in bounded domains).  

\begin{thm
} 
\label{bounded_chi}
Let $(c,\chi)$ be the solution to \eqref{cell} as constructed in Theorem \ref{cellteo}. If \eqref{C10extrastrong} holds, then
\begin{equation}\label{corr1} \min_{|y|\leq R } \chi(y)+\frac{1}{|x|^\rho}-\frac{1}{R^\rho}\leq \chi (x)\leq  \max_{|y|\leq R } \chi(y)-\frac{1}{|x|^\rho}+\frac{1}{R^\rho}\qquad \forall |x|\geq R.\end{equation}
In particular,  $\chi\in L^\infty(\R^N)$.
\end{thm}
\begin{proof}
First of all observe that $\chi(x)-ct$ is a solution of  \begin{equation}\label{para} u_t+F(x, Du,D^2u)=0, \qquad \forall x\in\R^N, \ t\in (-\infty, +\infty).\end{equation} 
For $R>0$ given by assumption \eqref{C10extrastrong} 
define  $w(x):=R^{-\rho}-|x|^{-\rho}$.  
Then $w\geq 0$ for $|x|\geq R$ and 
\begin{multline}
\label{boh}
\sup_{a\in A}\{ tr\, a(x,\a) D^2 w(x)+ b(x,\a)\cdot Dw(x)  \} =\\
 \frac{\rho}{|x|^{\rho+2}}\sup_{a\in A}\left\{ tr\, a(x,\a) + b(x,\a)\cdot x-\frac{\rho+2}{|x|^{2} }|\sigma(x,\a)\cdot x|^2\right\} 
 \leq 
  -2|c|-\|l\|_\infty \qquad |x|>R.
\end{multline}
Fix now $h\in (0,1)$ and consider $t_h<0$ such that \begin{equation}\label{corr2} 
-|c|t_h\geq \max_{|y|\geq R} \left(\chi(y)-h\frac{|y|^2}{2}\right)-\max_{|z|\leq R} \chi(z), \end{equation}
where the first maximum exists due to \eqref{growthcorr}. 
Define the function 
\[
v(t,x):= \max_{|z|\leq R} \chi(z)+h\frac{|x|^2}{2}+R^{-\rho}-|x|^{-\rho}-2|c| t\qquad t_h\leq t\leq 0, \ x\in\R^N.
\]
We claim that $\chi(x)-ct\leq v(t,x)$ for every $t\in [t_h,0]$ and every $x$ 
 with $|x|\geq R$.
First of all observe that the inequality holds at $t=t_h$. 
Indeed, 
by our choice of $t_h$, for $|x|\geq R$ we get 
\[v(t_h, x)\geq  \max_{|z|\leq R} \chi(z)+h\frac{|x|^2}{2}-2|c| t_h\geq \chi(x) -|c|t_h\geq \chi(x)-c t_h.\] 
Moreover, if $|x|=R$ and $t\leq 0$, then $v(t,x)\geq \max_{|z|\leq R} \chi(z)- |c| t\geq \chi(x)-ct$.
Now assume by contradiction that $v(s,y)\leq \chi(y)-cs$ for some $|y|\geq R$ and $s\in [t_h, 0]$. Then, using again \eqref{growthcorr},
 $\max_{|y|\geq R,  s\in[t_h,0]} \chi(y)-cs-v(s,y)= \chi(x)-ct-v(t,x)>0$
 for some $|x|>R$ and $t\in (t_h, 0]$. 
From \eqref{boh} we get 
\begin{eqnarray*} 
v_t(t,x)+F(x, Dv(t,x), D^2v(t,x))&\geq & -2|c|-\left(h+\frac{\rho}{|x|^{\rho+2}}\right)\sup_{a\in A}\{ tr\, a(x,\a)+ b(x,\a)\cdot x \} -\|l\|_\infty\\ 
& \geq & -2|c| +2|c|+\|l\|_\infty +h|x|^{\rho+2}   \frac{2|c|+\|l\|_\infty}{\rho}-\|l\|_\infty\\ &\geq & h|x|^{\rho+2}   \frac{2|c|+\|l\|_\infty}{\rho}>0 ,
\end{eqnarray*} 
which 
contradicts the fact that $\chi(x)-ct$ is a subsolution to \eqref{para}. 

So, in particular, $\chi(x)\leq v(0,x)$, which gives  the inequality on the right  of \eqref{corr1} after letting $h\to 0$. The inequality on the left is obtained similarly, by considering $w(t,x)=  \min_{|z|\leq R} \chi(z)-h\frac{|x|^2}{2}-R^{-\rho}+|x|^{-\rho}+2|c| t$. 
\end{proof} 
\begin{rem}
\upshape \label{remell} 
Assume the matrix $a$ in the operator $F$ has a positive lower bound on the minimal eigenvalue, therefore strenghtening \eqref{C3} to
\begin{equation}
\label{C3*}
\xi^T a(x,\a) \xi\geq \lambda|\xi|^2 \quad \forall \xi,x\in\R^N,
\end{equation}
for some $\lambda>0$. Note that this implies the first inequality in the uniform ellipticity condition \eqref{unell} for both possible forms of $F$
 \eqref{fnuova}. 
Then the conclusions of Theorem \ref{bounded_chi} remain true if condition \eqref {C10extrastrong} is replaced by the weaker assumption
\begin{equation}
\label{C10lessstrong}
\sup_{a\in A}\{ tr\, a(x,\a) + b(x,\a)\cdot x  -\lambda(2+\rho)\}  \leq  -\frac{2|c|+\|l\|_\infty}{\rho} |x|^{2+\rho}     \quad\text{ for } |x|\geq R.
\end{equation}
In fact, \eqref{C3*} implies $|\sigma(x,\a)\xi|\geq \lambda|\xi|$, which can be used in \eqref{boh} with \eqref{C10lessstrong} to get the same conclusion.
\end{rem}

\end{document}